\documentclass[11pt]{article}%
\usepackage{amsmath}
\usepackage{amsfonts}
\usepackage{amssymb}
\usepackage{graphicx}
\usepackage{version}%
\providecommand{\U}[1]{\protect\rule{.1in}{.1in}}
\newtheorem{theorem}{Theorem}[section]

\newtheorem{corollary}[theorem]{Corollary}

\newtheorem{definition}[theorem]{Definition}

\newtheorem{lemma}[theorem]{Lemma}

\newtheorem{proposition}[theorem]{Proposition}
\newtheorem{remark}[theorem]{Remark}

\newenvironment{proof}[1][Proof]{\noindent\textbf{#1.} }{\ \rule{0.5em}{0.5em}}
\excludeversion{bib}
\begin{document}

\title{\textit{Asymptotics for the best Sobolev constants and their extremal
functions}}
\author{{\small \textbf{Grey Ercole\thanks{Corresponding author}\ \ and Gilberto de
Assis Pereira}}\\{\small \textit{Departamento de Matem\'{a}tica - Universidade Federal de Minas
Gerais}}\\{\small \textit{Belo Horizonte, MG, 30.123-970, Brazil.}}\\{\small grey@mat.ufmg.br, gilbertoapereira@yahoo.com.br}}
\maketitle

\begin{abstract}
Let $\Omega$ be a bounded and smooth domain of $\mathbb{R}^{N},$ $N\geq2,$ and
consider%
\[
\Lambda_{p}(\Omega):=\lim_{q\rightarrow\infty}\lambda_{q}(\Omega),\;p>N,
\]
where
\[
\lambda_{q}(\Omega):=\inf\left\{  \left\Vert \nabla u\right\Vert
_{L^{p}(\Omega)}^{p}:u\in W_{0}^{1,p}(\Omega)\text{ and }\left\Vert
u\right\Vert _{L^{q}(\Omega)}=1\right\}  ,\;q\geq1.
\]

In the first part of the paper we show that
\[
\Lambda_{p}(\Omega)=\min\left\{  \left\Vert \nabla u\right\Vert _{L^{p}%
(\Omega)}^{p}:u\in W_{0}^{1,p}(\Omega)\text{ and }\left\Vert u\right\Vert
_{L^{\infty}(\Omega)}=1\right\}  ,
\]
the minimum being achieved by a positive function $u_{p}=\lim_{q_{n}%
\rightarrow\infty}w_{q_{n}}$ (convergence in $W_{0}^{1,p}(\Omega)$ and also in
$C(\overline{\Omega})$), where $w_{q}$ denotes a positive minimizer of
$\lambda_{q}(\Omega).$ We also prove that any minimizer $u_{p}$ of
$\Lambda_{p}(\Omega)$ satisfies
\[
-\Delta_{p}u_{p}=u_{p}(x_{p})\Lambda_{p}(\Omega)\delta_{x_{p}},
\]
where $\delta_{x_{p}}$ is the Dirac delta distribution concentrated at the
unique maximum point of $\left\vert u_{p}\right\vert .$

In the second part of the paper we first prove that
\[
\lim_{p\rightarrow\infty}\Lambda_{p}(\Omega)^{\frac{1}{p}}=\dfrac
{1}{\left\Vert \rho\right\Vert _{L^{\infty}(\Omega)}},
\]
where $\rho$ denotes the distance function to the boundary $\partial\Omega.$
Then, we show that there exist $p_{n}\rightarrow\infty$ and $u_{\infty}\in
W_{0}^{1,\infty}(\Omega)$ such that: $u_{p_{n}}\rightarrow u_{\infty}$
uniformly in $\overline{\Omega},$ $0<u_{\infty}\leq\frac{\rho}{\left\Vert
\rho\right\Vert _{L^{\infty}(\Omega)}}$ in $\Omega,$ and%
\[
\left\{
\begin{array}
[c]{ll}%
\Delta_{\infty}u_{\infty}=0 & \text{in }\Omega\backslash\left\{  x_{\ast
}\right\} \\
u_{\infty}=\frac{\rho}{\left\Vert \rho\right\Vert _{L^{\infty}(\Omega)}} &
\text{on }\partial\Omega\cup\left\{  x_{\ast}\right\}  ,
\end{array}
\right.
\]
in the viscosity sense, where $x_{\ast}:=\lim x_{p_{n}}$ is a maximum point of
$\rho.$

We also prove that $x_{\ast}$ is the unique maximum point of $u_{\infty}$ and
give conditions on $\Omega,$ under which $u_{\infty}=\frac{\rho}{\left\Vert
\rho\right\Vert _{L^{\infty}(\Omega)}}.$\bigskip

\end{abstract}



\section{Introduction}

Let $p>1$ and let $\Omega$ be a bounded and smooth domain of $\mathbb{R}^{N},$
$N\geq2.$ It is well known that the Sobolev immersion $W_{0}^{1,p}%
(\Omega)\hookrightarrow L^{q}(\Omega)$ is compact if
\[
1\leq q<p^{\star}:=\left\{
\begin{array}
[c]{ll}%
\frac{Np}{N-p} & \text{if }1<p<N\\
\infty & \text{if }p\geq N.
\end{array}
\right.
\]
As consequence of this fact, for each $q\in\lbrack1,p^{\star})$ there exists
$w_{q}\in L^{q}(\Omega)$ such that $\left\Vert w_{q}\right\Vert _{q}=1$ and%
\begin{equation}
\lambda_{q}(\Omega):=\inf\left\{  \left\Vert \nabla u\right\Vert _{p}^{p}:u\in
W_{0}^{1,p}(\Omega)\text{ and }\left\Vert u\right\Vert _{q}=1\right\}
=\left\Vert \nabla w_{q}\right\Vert _{p}^{p}. \label{lambdq}%
\end{equation}
(Throughout this paper $\left\Vert \cdot\right\Vert _{s}$ denotes the standard
norm of $L^{s}(\Omega),$ $1\leq s\leq\infty.$)

The value $\lambda_{q}(\Omega)$ is, therefore, the best constant $c$ in the
Sobolev inequality%
\[
c\left\Vert u\right\Vert _{q}^{p}\leq\left\Vert \nabla u\right\Vert _{p}%
^{p};\text{ \ }u\in W_{0}^{1,p}(\Omega)
\]
and $w_{q}$ is a corresponding extremal (or minimizer) function.

The Euler-Lagrange formulation associated with the minimizing problem
(\ref{lambdq}) is
\begin{equation}
\left\{
\begin{array}
[c]{clll}%
-\Delta_{p}u & = & \lambda_{q}(\Omega)\left\vert u\right\vert ^{q-2}u &
\text{in }\Omega\\
u & = & 0 & \text{on }\partial\Omega
\end{array}
\right.  \label{EL}%
\end{equation}
where $\Delta_{p}u:=\operatorname{div}\left(  \left\vert \nabla u\right\vert
^{p-2}\nabla u\right)  $ is the $p$-Laplacian operator. It turns out that
$\left\vert w_{q}\right\vert $ is a nonnegative and nontrivial solution of
(\ref{EL}), since $\left\vert w_{q}\right\vert $ also minimizes $\lambda
_{q}(\Omega).$ Thus, the maximum principle (see \cite{Vazquez}) assures that
$w_{q}$ does not change sign in $\Omega.$

From now on, we denote by $w_{q}$ any positive extremal function of
$\lambda_{q}(\Omega).$ Therefore, such a function enjoys the following
properties
\[
\left\Vert w_{q}\right\Vert _{q}=1,\text{ }\left\Vert \nabla w_{q}\right\Vert
_{p}^{p}=\lambda_{q}(\Omega)\text{ \ and }\left\{
\begin{array}
[c]{clll}%
-\Delta_{p}w_{q} & = & \lambda_{q}(\Omega)w_{q}^{q-1} & \text{in }\Omega\\
w_{q} & = & 0 & \text{on }\partial\Omega\\
w_{q} & > & 0 & \text{in }\Omega.
\end{array}
\right.
\]

It can be checked (see \cite[Lemma 4.2]{GECCM}), as a simple application of
the H\"{o}lder inequality, that the function
\begin{equation}
q\in\lbrack1,p^{\star})\mapsto\lambda_{q}(\Omega)\left\vert \Omega\right\vert
^{\frac{p}{q}}\text{ } \label{monot}%
\end{equation}
is decreasing for any fixed $p>1,$ where here and from now on $\left\vert
D\right\vert $ denotes the Lebesgue volume of the set $D,$ i.e. $\left\vert
D\right\vert =\int_{D}\mathrm{d}x.$

The monotonicity of the function in (\ref{monot}) guarantees that%
\[
\Lambda_{p}(\Omega):=\lim_{q\rightarrow p^{\star}}\lambda_{q}(\Omega)
\]
is well defined and also that%
\[
0\leq\Lambda_{p}(\Omega)=\inf_{q\geq1}\left(  \lambda_{q}(\Omega)\left\vert
\Omega\right\vert ^{\frac{p}{q}}\right)  \left(  \lim_{q\rightarrow p^{\star}%
}\left\vert \Omega\right\vert ^{-\frac{p}{q}}\right)  \leq\lambda_{1}%
(\Omega)\left\vert \Omega\right\vert ^{p}\left(  \lim_{q\rightarrow p^{\star}%
}\left\vert \Omega\right\vert ^{-\frac{p}{q}}\right)  .
\]

It is known that
\begin{equation}
\Lambda_{p}(\Omega)=\left\{
\begin{array}
[c]{ll}%
S_{p} & \text{if }1<p<N\\
0 & \text{if }p=N,
\end{array}
\right.  \label{GamaoleqN}%
\end{equation}
where $S_{p}$ is the Sobolev constant: the best constant $S$ in the Sobolev
inequality
\[
S\left\Vert u\right\Vert _{L^{p^{\star}}(\mathbb{R}^{N})}^{p}\leq\left\Vert
\nabla u\right\Vert _{L^{p}(\mathbb{R}^{N})}^{p},\text{ \ \ }u\in W_{0}%
^{1,p}(\mathbb{R}^{N}).
\]
It is explicitly given by (see \cite{Auban, Talenti})
\begin{equation}
S_{p}:=\pi^{\frac{p}{2}}N\left(  \frac{N-p}{p-1}\right)  ^{p-1}\left(
\frac{\Gamma(N/p)\Gamma(1+N-N/p)}{\Gamma(1+N/2)\Gamma(N)}\right)  ^{\frac
{p}{N}} \label{SNp}%
\end{equation}
where $\Gamma(t)=\int_{0}^{\infty}s^{t-1}e^{-s}ds$ is the Gamma Function.

The case $1<p<N$ in (\ref{GamaoleqN}) can be seen in \cite{GEJMAA}, whereas
the case $p=N$ is consequence of the following result proved in
\cite{RenWei-p}
\[
\lim_{q\rightarrow\infty}q^{N-1}\lambda_{q}(\Omega)=\frac{N^{2N-1}\omega_{N}%
}{(N-1)^{N-1}}e^{N-1},
\]
where
\begin{equation}
\omega_{N}=\frac{\pi^{\frac{N}{2}}}{\Gamma(\frac{N}{2}+1)} \label{volB1}%
\end{equation}
is the volume of the unit ball $B_{1}.$ (From now on $B_{\rho}$ denotes the
ball centered at the origin with radius $\rho$).

As we can see from (\ref{GamaoleqN}) the value $\Lambda_{p}(\Omega)$ does not
depend on $\Omega,$ when $1<p\leq N.$ This property does not hold if $p>N.$
Indeed, by using a simple scaling argument one can show that%
\[
\Lambda_{p}(B_{R})=\Lambda_{p}(B_{1})R^{N-p}.
\]

In the first part of this paper, developed in Section \ref{part1}, we consider
a general bounded domain $\Omega$ and $p>N$ and show that
\begin{equation}
\Lambda_{p}(\Omega)=\inf\left\{  \left\Vert \nabla u\right\Vert _{p}^{p}:u\in
W_{0}^{1,p}(\Omega)\text{ \ and \ }\left\Vert u\right\Vert _{\infty
}=1\right\}  . \label{inf}%
\end{equation}
Thus, $\Lambda_{p}(\Omega)$ is the best constant associated with the (compact)
Sobolev immersion
\[
W_{0}^{1,p}(\Omega)\hookrightarrow C(\overline{\Omega}),\text{ \ \ }p>N,
\]
in the sense that it is the sharp value for a constant $c$ satisfying%
\[
c\left\Vert u\right\Vert _{\infty}^{p}\leq\left\Vert \nabla u\right\Vert
_{p}^{p},\text{ \ for all\ }u\in W_{0}^{1,p}(\Omega).
\]

We also show that there exists $q_{n}\rightarrow\infty$ such that $w_{q_{n}}$
converges strongly, in both Banach spaces $C(\overline{\Omega})$
and\ $W_{0}^{1,p}(\Omega),$ to a positive function $u_{p}$ satisfying
$\left\Vert u_{p}\right\Vert _{\infty}=1.$ Moreover, we prove that this
function attains the infimum at (\ref{inf}):%
\begin{equation}
\left\Vert \nabla u_{p}\right\Vert _{p}^{p}=\Lambda_{p}(\Omega)=\min\left\{
\left\Vert \nabla u\right\Vert _{p}^{p}:u\in W_{0}^{1,p}(\Omega)\text{ \ and
\ }\left\Vert u\right\Vert _{\infty}=1\right\}  . \label{min}%
\end{equation}

However, our main result in Section \ref{part1} is the complete
characterization of the minimizers in (\ref{min}), which we call extremal
functions of $\Lambda_{p}(\Omega)$ and denote by $u_{p}.$ More precisely, we
prove that if $u_{p}\in W_{0}^{1,p}(\Omega)$ is such that%
\[
\left\Vert u_{p}\right\Vert _{\infty}=1\text{ \ and \ }\left\Vert \nabla
u_{p}\right\Vert _{p}^{p}=\Lambda_{p}(\Omega)
\]
then $u_{p}$ does not change sign in $\Omega,$ attains its sup norm at a
unique point $x_{p}$ and satisfies the equation%
\[
-\Delta_{p}u_{p}=u_{p}(x_{p})\Lambda_{p}(\Omega)\delta_{x_{p}}%
\]
where $\delta_{x_{p}}$ is the Dirac delta distribution concentrated at
$x_{p}.$

In the particular case where $\Omega=B_{R},$ a ball of radius $R,$ we show
that%
\begin{equation}
\Lambda_{p}(B_{R})=\frac{N\omega_{N}}{R^{p-N}}\left(  \frac{p-N}{p-1}\right)
^{p-1} \label{res1}%
\end{equation}
and that
\begin{equation}
\lim_{q\rightarrow\infty}w_{q}(\left\vert x\right\vert )=u_{p}(\left\vert
x\right\vert ):=1-\left(  \frac{\left\vert x\right\vert }{R}\right)
^{\frac{p-N}{p-1}};\text{ \ }0\leq\left\vert x\right\vert \leq R, \label{res2}%
\end{equation}
where $w_{q}(\left\vert \cdot\right\vert )$ is the positive extremal function
of $\lambda_{q}(B_{R}).$ Moreover, we prove that the function $u_{p}$ defined
in (\ref{res2}) is the unique minimizer of $\Lambda_{p}(B_{R}).$ Since
$x_{p}=0$, our main result in Section \ref{part1} implies that%
\[
-\Delta_{p}u_{p}=\Lambda_{p}(B_{R})\delta_{0}.
\]

It is convenient to recall the following consequence of Theorem 2.E of
\cite{Talenti1}, due to Talenti:
\begin{equation}
N(\omega_{N})^{\frac{p}{N}}\left(  \frac{p-N}{p-1}\right)  ^{p-1}\left\vert
\Omega\right\vert ^{1-\frac{p}{N}}\left\Vert u\right\Vert _{\infty}^{p}%
\leq\left\Vert \nabla u\right\Vert _{p}^{p},\text{ \ for all }u\in W_{0}%
^{1,p}(\Omega). \label{talenti}%
\end{equation}
We emphasize that, in view of (\ref{inf}), this inequality allows one to
conclude that
\begin{equation}
N(\omega_{N})^{\frac{p}{N}}\left(  \frac{p-N}{p-1}\right)  ^{p-1}\left\vert
\Omega\right\vert ^{1-\frac{p}{N}}\leq\Lambda_{p}(\Omega).
\label{lowerboundbc}%
\end{equation}

Note that when $\Omega=B_{R}$ the left-hand side of (\ref{lowerboundbc})
coincides with the right-hand side of (\ref{res1}). Thus, equality in
(\ref{talenti}) holds when $\Omega$ is a ball and $u$ is a scalar multiple of
the function defined in (\ref{res2}), as pointed out in \cite{Talenti1}. In
this paper we show that if $\Omega$ is not a ball, then the inequality in
(\ref{talenti}) has to be strict. This fact was not observed in
\cite{Talenti1}.

We remark that (\ref{res1}) provides the following upper bound to $\Lambda
_{p}(\Omega):$
\[
\Lambda_{p}(\Omega)\leq N(\omega_{N})^{\frac{p}{N}}\left(  \frac{p-N}%
{p-1}\right)  ^{p-1}\left\vert B_{R_{\Omega}}\right\vert ^{1-\frac{p}{N}},
\]
where $R_{\Omega}$ denotes the inradius of $\Omega,$ that is, the radius of
the largest ball inscribed in $\Omega.$ We use the bounds (\ref{lowerboundbc})
and (\ref{res1}) and the explicit expression of $S_{p}$ in (\ref{SNp}) to
conclude that the function $p\mapsto\Lambda_{p}(\Omega)$ is continuous at
$p=N.$

In the second part, developed in Section \ref{part2}, we study the asymptotic
behavior, as $p\rightarrow\infty,$ of the pair $(\Lambda_{p}(\Omega),u_{p}),$
where $u_{p}\in W_{0}^{1,p}(\Omega)$ will denote a positive extremal function
of $\Lambda_{p}(\Omega).$ First we prove that
\[
\lim_{p\rightarrow\infty}\Lambda_{p}(\Omega)^{\frac{1}{p}}=\frac{1}{\left\Vert
\rho\right\Vert _{\infty}},
\]
where
\[
\rho(x):=\inf_{y\in\partial\Omega}\left\vert y-x\right\vert ,\text{ \ }%
x\in\Omega,
\]
is the distance function to the boundary. We recall the well-known fact:
\begin{equation}
\frac{1}{\left\Vert \rho\right\Vert _{\infty}}=\min\left\{  \frac{\left\Vert
\nabla\phi\right\Vert _{\infty}}{\left\Vert \phi\right\Vert _{\infty}}:\phi\in
W_{0}^{1,\infty}(\Omega)\backslash\left\{  0\right\}  \right\}  .
\label{minRay}%
\end{equation}

Then, we prove that there exist a sequence $p_{n}\rightarrow\infty,$ a point
$x_{\ast}\in\Omega$ and a function $u_{\infty}\in W_{0}^{1,\infty}(\Omega)\cap
C(\overline{\Omega})$ such that: $x_{p_{n}}\rightarrow x_{\ast},$ $\left\Vert
\rho\right\Vert _{\infty}=\rho(x_{\ast}),$ $u_{\infty}\leq\frac{\rho
}{\left\Vert \rho\right\Vert _{\infty}}$ and $u_{p_{n}}\rightarrow u_{\infty
},$ uniformly in $\overline{\Omega}$ and strongly in $W_{0}^{1,r}(\Omega)$ for
all $r>N.$ Moreover, $x_{\ast}$ is the unique maximum point of $u_{\infty},$
this function is also a minimizer of (\ref{minRay}) and satisfies%
\[
\left\{
\begin{array}
[c]{ll}%
\Delta_{\infty}u_{\infty}=0 & \text{in }\Omega\backslash\left\{  x_{\ast
}\right\} \\
u_{\infty}=\frac{\rho}{\left\Vert \rho\right\Vert _{\infty}} & \text{on
}\partial\left(  \Omega\backslash\left\{  x_{\ast}\right\}  \right)  =\left\{
x_{\ast}\right\}  \cup\partial\Omega
\end{array}
\right.
\]
in the viscosity sense, where $\Delta_{\infty}$ denotes the well-known
$\infty$-Laplacian operator (see \cite{BDM, Cr, Lq}), defined formally by
\[
\Delta_{\infty}u:=\sum_{i,j=1}^{N}\frac{\partial u}{\partial x_{i}}%
\frac{\partial u}{\partial x_{j}}\frac{\partial^{2}u}{\partial x_{i}\partial
x_{j}}.
\]

Still in Section \ref{part2} \ we characterize the domains $\Omega$ for which
$u_{\infty}=\frac{\rho}{\left\Vert \rho\right\Vert _{\infty}}$\ in
$\overline{\Omega}$ and show that each maximum point of the distance function
$\rho$ gives rise to a minimizer of (\ref{minRay}). We then use this latter
fact to conclude that if $\Omega$ is an annulus, then there exist infinitely
many positive and nonradial minimizers of (\ref{minRay}).

\section{$\Lambda_{p}(\Omega)$ and its extremal functions\label{part1}}

In this section, $p>N\geq2$ and $\Omega$ denotes a bounded and smooth domain
of $\mathbb{R}^{N}.$ We recall the well-known Morrey's inequality
\[
\left\Vert u\right\Vert _{C^{0,\gamma}(\overline{\Omega})}\leq C\left\Vert
\nabla u\right\Vert _{L^{p}(\Omega)},\text{ \ for all }u\in W^{1,p}(\Omega),
\]
where $\gamma:=1-\frac{N}{p}$ and $C$ depends only on $\Omega,$ $p$ and $N.$
This inequality implies immediately that the immersion $W_{0}^{1,p}%
(\Omega)\hookrightarrow C(\overline{\Omega})$ is compact.

Let us also recall that%
\[
\Lambda_{p}(\Omega):=\lim_{q\rightarrow\infty}\lambda_{q}(\Omega)
\]
where $\lambda_{q}(\Omega)$ is defined in (\ref{lambdq}).

\begin{theorem}
\label{teomain}There holds%
\begin{equation}
\Lambda_{p}(\Omega)=\inf\left\{  \left\Vert \nabla u\right\Vert _{p}^{p}:u\in
W_{0}^{1,p}(\Omega)\text{ \ and \ }\left\Vert u\right\Vert _{\infty
}=1\right\}  . \label{limit}%
\end{equation}

\end{theorem}

\begin{proof}
Let
\[
\mu:=\inf\left\{  \left\Vert \nabla u\right\Vert _{p}^{p}:u\in W_{0}%
^{1,p}(\Omega)\text{ \ and \ }\left\Vert u\right\Vert _{\infty}=1\right\}  .
\]

Let us take $u\in W_{0}^{1,p}(\Omega)$\ such that $\left\Vert u\right\Vert
_{\infty}=1.$ Since $\lim_{q\rightarrow\infty}\left\Vert u\right\Vert
_{q}=\left\Vert u\right\Vert _{\infty}=1$ we have%
\[
\Lambda_{p}(\Omega)=\lim_{q\rightarrow\infty}\lambda_{q}(\Omega)\leq
\lim_{q\rightarrow\infty}\frac{\left\Vert \nabla u\right\Vert _{p}^{p}%
}{\left\Vert u\right\Vert _{q}^{p}}=\left\Vert \nabla u\right\Vert _{p}^{p},
\]
implying that $\Lambda_{p}(\Omega)\leq\mu.$

Now, for each $q\geq1$ let $w_{q}$ be a positive extremal function of
$\lambda_{q}(\Omega).$ Since
\[
\mu\leq\left\Vert \nabla(w_{q}/\left\Vert w_{q}\right\Vert _{\infty
})\right\Vert _{p}^{p}=\frac{\lambda_{q}(\Omega)}{\left\Vert w_{q}\right\Vert
_{\infty}^{p}},
\]
in order to verify that $\mu\leq\Lambda_{p}(\Omega)$ we need only check that
\begin{equation}
\lim_{q\rightarrow\infty}\left\Vert w_{q}\right\Vert _{\infty}=1.
\label{norm1}%
\end{equation}
Since $1=\left\Vert w_{q}\right\Vert _{q}\leq\left\vert \Omega\right\vert
^{\frac{1}{q}}\left\Vert w_{q}\right\Vert _{\infty},$ $\left\Vert \nabla
w_{q}\right\Vert _{p}^{p}=\lambda_{q}(\Omega)$ and $\Lambda_{p}(\Omega)\leq
\mu$ we have%
\[
\left\vert \Omega\right\vert ^{-\frac{p}{q}}\leq\left\Vert w_{q}\right\Vert
_{\infty}^{p}\leq\frac{\left\Vert \nabla w_{q}\right\Vert _{p}^{p}}{\mu}%
\leq\frac{\lambda_{q}(\Omega)}{\Lambda_{p}(\Omega)},
\]
which leads to (\ref{norm1}), after making $q\rightarrow\infty.$
\end{proof}

Taking into account (\ref{limit}), we make the following definition:

\begin{definition}
We say that $v\in W_{0}^{1,p}(\Omega)$ is an extremal function of $\Lambda
_{p}(\Omega)$ iff
\[
\left\Vert \nabla v\right\Vert _{p}^{p}=\Lambda_{p}(\Omega)\text{ \ and
\ }\left\Vert v\right\Vert _{\infty}=1.
\]

\end{definition}

In the sequel we show that an extremal function of $\Lambda_{p}(\Omega)$ can
be obtained as the limit of $w_{q_{n}}$ for some $q_{n}\rightarrow\infty,$
where $w_{q_{n}}$ denotes the extremal function of $\lambda_{q_{n}}(\Omega).$

\begin{theorem}
\label{teomain4}There exists $q_{n}\rightarrow\infty$ and a nonnegative
function $w\in W_{0}^{1,p}(\Omega)\cap C(\overline{\Omega})$ such that
$w_{q_{n}}\rightarrow w$ strongly in $C(\overline{\Omega})$ and also in
$W_{0}^{1,p}(\Omega).$ Moreover, $w$ is an extremal function of $\Lambda
_{p}(\Omega).$
\end{theorem}

\begin{proof}
Since $w_{q}$ is uniformly bounded in $W_{0}^{1,p}(\Omega)$ and also in
$C^{0,1-\frac{N}{p}}(\overline{\Omega})$ there exist $q_{n}\rightarrow\infty$
and a nonnegative function $w\in W_{0}^{1,p}(\Omega)\cap C(\overline{\Omega})$
such that $w_{q_{n}}\rightarrow w$ weakly in $W_{0}^{1,p}(\Omega)$ and
strongly in $C(\overline{\Omega}).$ Thus, $\left\Vert w\right\Vert _{\infty
}=\lim\left\Vert w_{q_{n}}\right\Vert _{\infty}=1$ (because of (\ref{norm1}))
and hence%
\[
\Lambda_{p}(\Omega)\leq\left\Vert \nabla w\right\Vert _{p}^{p}\leq
\liminf\left\Vert \nabla w_{q_{n}}\right\Vert _{p}^{p}=\lim\lambda_{q_{n}%
}(\Omega)=\Lambda_{p}(\Omega).
\]
This implies that $\Lambda_{p}(\Omega)=\lim\left\Vert \nabla w_{q_{n}%
}\right\Vert _{p}^{p}=\left\Vert \nabla w\right\Vert _{p},$ so that $w_{q_{n}%
}\rightarrow w$ strongly in $W_{0}^{1,p}(\Omega)$ and also that $w$ is an
extremal function of $\Lambda_{p}(\Omega).$
\end{proof}

\begin{remark}
As we will see in the sequel, any nonnegative extremal function of
$\Lambda_{p}(\Omega)$ must be strictly positive in $\Omega.$
\end{remark}

We recall a well-known fact: $(-\Delta_{p})^{-1}:W^{-1,p^{\prime}}%
(\Omega)\mapsto W_{0}^{1,p}(\Omega)$ is bijective. Thus, if $p>N$ the
equation
\begin{equation}
-\Delta_{p}u=c\delta_{y} \label{eqfunc}%
\end{equation}
has a unique solution $u\in W_{0}^{1,p}(\Omega)$ for each fixed $y\in\Omega$
and $c\in\mathbb{R}.$ Note that if $p>N$ then $\delta_{y}\in W_{0}%
^{-1,p^{\prime}}(\Omega),$ since
\[
\left\vert \delta_{y}(\phi)\right\vert =\left\vert \phi(y)\right\vert
\leq\left\Vert \phi\right\Vert _{\infty}\leq\Lambda_{p}(\Omega)^{-\frac{1}{p}%
}\left\Vert \nabla\phi\right\Vert _{p},\text{ \ for all }\phi\in W_{0}%
^{1,p}(\Omega).
\]

The equation in (\ref{eqfunc}) is to be interpreted in sense of the
distributions:%
\[
\int_{\Omega}\left\vert \nabla u\right\vert ^{p-2}\nabla u\cdot\nabla
\phi\mathrm{d}x=c\phi(y),\text{ for all }\phi\in W_{0}^{1,p}(\Omega).
\]

\begin{theorem}
\label{teomain3}Let $u_{p}\in W_{0}^{1,p}(\Omega)$ be an extremal function of
$\Lambda_{p}(\Omega)$ and let $x_{p}\in\Omega$ be such that%
\[
\left\vert u_{p}(x_{p})\right\vert =\left\Vert u_{p}\right\Vert _{\infty}=1.
\]
We claim that
\end{theorem}

\begin{enumerate}
\item[(i)] $-\Delta_{p}u_{p}=u_{p}(x_{p})\Lambda_{p}(\Omega)\delta_{x_{p}}$
\ in \ $\Omega,$

\item[(ii)] $x_{p}$ is the unique global maximum point of $\left\vert
u_{p}\right\vert ,$

\item[(iii)] $u_{p}$ does not change sign in $\Omega,$ and

\item[(iv)] for each $0<t<1,$ there exists $\alpha_{t}\in(0,1)$ such that
$u_{p}\in C^{1,\alpha_{t}}(\overline{E_{t}}),$ where $E_{t}=\left\{
x\in\Omega:0<\left\vert u_{p}(x)\right\vert <t\right\}  .$
\end{enumerate}

\begin{proof}
For the sake of simplicity, we will assume throughout this proof that
$u_{p}(x_{p})=1$ (otherwise, if $u_{p}(x_{p})=-1,$ we replace $u_{p}$ by
$-u_{p}.$)

Let $v\in W_{0}^{1,p}(\Omega)$ be such that
\[
-\Delta_{p}v=\Lambda_{p}(\Omega)\delta_{x_{p}}\text{ \ in \ }\Omega.
\]

Since $u_{p}(x_{p})=1,$%
\begin{equation}
\Lambda_{p}(\Omega)=\int_{\Omega}\left\vert \nabla v\right\vert ^{p-2}\nabla
v\cdot\nabla u_{p}\mathrm{d}x\leq\int_{\Omega}\left\vert \nabla v\right\vert
^{p-1}\left\vert \nabla u_{p}\right\vert \mathrm{d}x. \label{aux18}%
\end{equation}
Hence, since $\Lambda_{p}(\Omega)=\left\Vert \nabla u_{p}\right\Vert _{p}^{p}$
and
\begin{equation}
\left\Vert \nabla v\right\Vert _{p}^{p}=\int_{\Omega}\left\vert \nabla
v\right\vert ^{p-2}\nabla v\cdot\nabla v\mathrm{d}x=\Lambda_{p}(\Omega
)v(x_{p}) \label{aux19}%
\end{equation}
we apply H\"{o}lder inequality to (\ref{aux18}) in order to get
\begin{equation}
\int_{\Omega}\left\vert \nabla v\right\vert ^{p-1}\left\vert \nabla
u_{p}\right\vert \mathrm{d}x\leq\left\Vert \nabla v\right\Vert _{p}%
^{p-1}\left\Vert \nabla u_{p}\right\Vert _{p}=\left(  \Lambda_{p}%
(\Omega)v(x_{p})\right)  ^{\frac{p-1}{p}}\Lambda_{p}(\Omega)^{\frac{1}{p}%
}=\Lambda_{p}(\Omega)\left(  v(x_{p})\right)  ^{\frac{p-1}{p}}. \label{aux12}%
\end{equation}

It follows from (\ref{aux18}) and (\ref{aux12}) that $1\leq v(x_{p}%
)\leq\left\Vert v\right\Vert _{\infty}.$

On the other hand, (\ref{limit}) and (\ref{aux19}) yield%
\begin{equation}
\Lambda_{p}(\Omega)\leq\frac{\left\Vert \nabla v\right\Vert _{p}^{p}%
}{\left\Vert v\right\Vert _{\infty}^{p}}=\frac{\Lambda_{p}(\Omega)v(x_{p}%
)}{\left\Vert v\right\Vert _{\infty}^{p}}\leq\frac{\Lambda_{p}(\Omega
)}{\left\Vert v\right\Vert _{\infty}^{p-1}}. \label{aux11}%
\end{equation}
Hence, $v(x_{p})\leq\left\Vert v\right\Vert _{\infty}\leq1$ and then we
conclude that
\begin{equation}
1=v(x_{p})=\left\Vert v\right\Vert _{\infty}. \label{aux13}%
\end{equation}

Combining (\ref{aux13}) with (\ref{aux11}) we obtain
\[
\Lambda_{p}(\Omega)=\left\Vert \nabla v\right\Vert _{p}^{p},
\]
showing that $v$ is an extremal function of $\Lambda_{p}(\Omega).$

In order to prove that $u_{p}=v$ we combine (\ref{aux13}) with (\ref{aux12})
and (\ref{aux18}) to get
\begin{equation}
\Lambda_{p}(\Omega)=\int_{\Omega}\left\vert \nabla v\right\vert ^{p-2}\nabla
v\cdot\nabla u_{p}\mathrm{d}x=\int_{\Omega}\left\vert \nabla v\right\vert
^{p-1}\left\vert \nabla u_{p}\right\vert \mathrm{d}x=\left\Vert \nabla
v\right\Vert _{p}^{p-1}\left\Vert \nabla u_{p}\right\Vert _{p}. \label{aux15}%
\end{equation}
The third equality in (\ref{aux15}) is exactly the case of an equality in the
H\"{o}lder inequality. It means that%
\begin{equation}
\left\vert \nabla v\right\vert =\left\vert \nabla u_{p}\right\vert \text{
\ \ a. e. in }\Omega. \label{aux14}%
\end{equation}
(Note that $\left\Vert \nabla v\right\Vert _{p}=\left\Vert \nabla
u_{p}\right\Vert _{p}.$)

We still obtain from (\ref{aux15}) that%
\[
0=\int_{\Omega}\left\vert \nabla v\right\vert ^{p-2}\left(  \left\vert \nabla
v\right\vert \left\vert \nabla u_{p}\right\vert -\nabla v\cdot\nabla
u_{p}\right)  \mathrm{d}x.
\]
Since $\left\vert \nabla v\right\vert \left\vert \nabla u_{p}\right\vert
\geq\nabla v\cdot\nabla u_{p}$ this yields%
\begin{equation}
\nabla v\cdot\nabla u_{p}=\left\vert \nabla v\right\vert \left\vert \nabla
u_{p}\right\vert \text{ \ \ a. e. in }\Omega. \label{aux16}%
\end{equation}
Note that this equality occurs even at the points where $\left\vert \nabla
v\right\vert ^{p-2}=0.$

It follows from (\ref{aux16}) and (\ref{aux14}) that
\[
\nabla v=\nabla u_{p}\text{ \ \ a. e. in }\Omega,
\]
implying that $\left\Vert \nabla(v-u_{p})\right\Vert _{p}=0.$ Since both $v$
and $u_{p}$ belong to $W_{0}^{1,p}(\Omega)$ we conclude that
\[
v=u_{p}\text{ \ \ a. e. in }\Omega
\]
so that $-\Delta_{p}u_{p}=\Lambda_{p}(\Omega)\delta_{x_{p}}.$ Thus, the proof
of (i) is complete.

The claim (ii) follows directly from (i). In fact, another global maximum
point, say $x_{1},$ would lead to the following absurd: $\Lambda_{p}%
(\Omega)\delta_{x_{p}}=-\Delta_{p}u_{p}=\Lambda_{p}(\Omega)\delta_{x_{1}}.$

Let us prove (iii). First we observe that $u_{p}\geq0$ in $\Omega.$ This is a
consequence of the weak comparison principle since%
\[
\int_{\Omega}\left\vert \nabla u_{p}\right\vert ^{p-2}\nabla u_{p}\cdot
\nabla\phi\mathrm{d}x=\Lambda_{p}(\Omega)\phi(x_{p})\geq0,\text{ for all }%
\phi\in W_{0}^{1,p}(\Omega),\text{ \ }\phi\geq0.
\]

Now, we argue that $u_{p}$ is $p$-harmonic in $\Omega\backslash\{x_{p}\}$.
Indeed, for each ball $B\subset\Omega\backslash\{x_{p}\}$ and each $\phi\in
W_{0}^{1,p}(B)\subset W_{0}^{1,p}(\Omega)$ (here we are considering $\phi=0$
in $\Omega\backslash B$) we have%
\[
\int_{B}\left\vert \nabla u_{p}\right\vert ^{p-2}\nabla u_{p}\cdot\nabla
\phi\mathrm{d}x=\int_{\Omega}\left\vert \nabla u_{p}\right\vert ^{p-2}\nabla
u_{p}\cdot\nabla\phi\mathrm{d}x=\Lambda_{p}(\Omega)\phi(x_{p})=0,
\]
implying that $u_{p}$ is $p$-harmonic in $B.$

Let us consider the following subset $Z:=\left\{  x\in\Omega:u_{p}%
(x)=0\right\}  .$ Of course, $Z$ is closed in $\Omega.$ Moreover, $Z$ is also
open in $\Omega.$ In fact, if $z\in Z$ then $z\in B$ for some ball
$B\subset\Omega\backslash\{x_{p}\}.$ Since $u_{p}$ is nonnegative in $B$ we
can conclude that $u_{p}$ restricted to $B$ assumes its minimum value $0$ at
$z\in B.$ Since $u_{p}$ is $p$-harmonic in $B$ it must assume its minimum
value only on the boundary $\partial B,$ unless it is constant on $B$ (see
\cite{Lindvist}). So, we conclude that $u_{p}$ is null in $B,$ proving that
$B\subset Z.$ Since $\Omega$ is connected (because it is a domain) the only
possibility to $Z$ is to be empty. This fact implies that $u_{p}>0$ in
$\Omega.$

In order to prove (iv) let us take $0<t<1$ and consider the set $E_{t}%
=\left\{  x\in\Omega:0<u_{p}(x)<t\right\}  ,$ which is open, since $u_{p}$ is
continuous. We remark that $u_{p}$ is constant on $\partial E_{t}.$ Moreover,
by following the reasoning made in the proof of the third claim, $u_{p}$ is
$p$-harmonic in $E_{t}$ because this set is away from $\left\{  x_{p}\right\}
$ (recall that $t<u_{p}(x_{p})$). Thus, $u_{p}$ is constant on $\partial
E_{t}$ and satisfies $-\Delta_{p}u_{p}=0$ in $E_{t}.$ This fact allows us to
apply the regularity result of Lieberman (see \cite[Theorem 1]{Lieberman}) to
each connected component of $E_{t}$ to conclude that there exists $\alpha
_{t}\in(0,1)$ such that $u_{p}\in C^{1,\alpha_{t}}(\overline{E_{t}}).$
\end{proof}

The next theorem is contained in Theorem 2.E of \cite{Talenti1}.

\begin{theorem}
\label{teomain2}Let $R>0.$ Consider the function
\begin{equation}
u_{p}(\left\vert x\right\vert ):=1-\left(  \frac{\left\vert x\right\vert }%
{R}\right)  ^{\frac{p-N}{p-1}};\text{ \ }0\leq\left\vert x\right\vert \leq R.
\label{wfunc}%
\end{equation}
One has,%
\begin{equation}
\left\Vert \nabla u_{p}\right\Vert _{p}^{p}=\frac{N\omega_{N}}{R^{p-N}}\left(
\frac{p-N}{p-1}\right)  ^{p-1}=\Lambda_{p}(B_{R}). \label{p>N}%
\end{equation}

\end{theorem}

\begin{proof}
We have%
\begin{align*}
\left\Vert \nabla u_{p}\right\Vert _{p}^{p}  &  =\int_{B_{R}}\left\vert \nabla
u_{p}(\left\vert x\right\vert )\right\vert ^{p}\mathrm{d}x\\
&  =N\omega_{N}\int_{0}^{R}r^{N-1}\left\vert u_{p}^{\prime}(r)\right\vert
^{p}\mathrm{d}r\\
&  =N\omega_{N}\left(  \frac{p-N}{p-1}\right)  ^{p}R^{-(\frac{p-N}{p-1})p}%
\int_{0}^{R}r^{N-1+(\frac{p-N}{p-1}-1)p}\mathrm{d}r\\
&  =N\omega_{N}\left(  \frac{p-N}{p-1}\right)  ^{p}R^{-(\frac{p-N}{p-1}%
)p}\frac{p-1}{p-N}R^{\frac{p-N}{p-1}}=\frac{N\omega_{N}}{R^{p-N}}\left(
\frac{p-N}{p-1}\right)  ^{p-1},
\end{align*}
which gives the first equality in (\ref{p>N}).

Of course, $u_{p}(\left\vert \cdot\right\vert )\in W_{0}^{1,p}(B_{R}).$ Since
$\left\Vert u_{p}\right\Vert _{\infty}=1,$ it follows from Theorem
\ref{teomain} that
\[
\Lambda_{p}(B_{R})\leq\left\Vert \nabla u_{p}\right\Vert _{p}^{p}%
=\frac{N\omega_{N}}{R^{p-N}}\left(  \frac{p-N}{p-1}\right)  ^{p-1}.
\]

On the other hand, it follows from (\ref{talenti}) that if $v\in W_{0}%
^{1,p}(B_{R})$ and $\left\Vert v\right\Vert _{\infty}=1$ then
\[
\frac{N\omega_{N}}{R^{p-N}}\left(  \frac{p-N}{p-1}\right)  ^{p-1}=N(\omega
_{N})^{\frac{p}{N}}\left(  \frac{p-N}{p-1}\right)  ^{p-1}\left\vert
B_{R}\right\vert ^{1-\frac{p}{N}}\leq\left\Vert \nabla v\right\Vert _{p}^{p}.
\]
Taking into account Theorem \ref{teomain}, this means that
\begin{equation}
\frac{N\omega_{N}}{R^{p-N}}\left(  \frac{p-N}{p-1}\right)  ^{p-1}\leq
\Lambda_{p}(B_{R}) \label{liminf}%
\end{equation}
and the proof is complete.
\end{proof}

\begin{corollary}
The following estimates for $\Lambda_{p}(\Omega)$ hold%
\begin{equation}
N(\omega_{N})^{\frac{p}{N}}\left(  \frac{p-N}{p-1}\right)  ^{p-1}\left\vert
\Omega\right\vert ^{1-\frac{p}{N}}\leq\Lambda_{p}(\Omega)\leq N(\omega
_{N})^{\frac{p}{N}}\left(  \frac{p-N}{p-1}\right)  ^{p-1}\left\vert
B_{R_{\Omega}}\right\vert ^{1-\frac{p}{N}},\text{ } \label{lowupp}%
\end{equation}
where $R_{\Omega}$ is the inradius of $\Omega$ (i.e. the radius of the largest
ball inscribed in $\Omega$).
\end{corollary}

\begin{proof}
The lower bound in (\ref{lowupp}) follows from (\ref{talenti}). Let
$B_{R_{\Omega}}(x_{0})\subset\Omega$ be a ball centered at a point $x_{0}%
\in\Omega$ with radius $R_{\Omega}.$ Since (it is easy to see)
\[
\Lambda_{p}(\Omega)\leq\Lambda_{p}(B_{R_{\Omega}}(x_{p}))=\Lambda
_{p}(B_{R_{\Omega}})
\]
we obtain the upper bound in (\ref{lowupp}) from (\ref{p>N}) with
$R=R_{\Omega}.$
\end{proof}

\begin{remark}
It follows from (\ref{lowupp}) that $\limsup_{p\rightarrow\infty}\Lambda
_{p}(\Omega)^{\frac{1}{p}}\leq R_{\Omega}^{-1}.$ As we will see in Section
\ref{part2}, $\Lambda_{p}(\Omega)^{\frac{1}{p}}$ increases as $p$ increases
and really converges to $R_{\Omega}^{-1}$ as $p\rightarrow\infty.$ This shows
that the upper bound in (\ref{lowupp}) gets asymptotically better as $p$ increases.
\end{remark}

\begin{corollary}
The equality in (\ref{talenti}) occurs for some $0\not \equiv u\in W_{0}%
^{1,p}(\Omega)$ if, and only if, $\Omega$ is ball.
\end{corollary}

\begin{proof}
When $\Omega=B_{R}$ the equality holds true in (\ref{talenti}) for the
function $u_{p}$ defined in (\ref{wfunc}), as (\ref{p>N}) shows. On the other
hand, if the equality in (\ref{talenti}) is verified for some $0\not \equiv
v\in W_{0}^{1,p}(\Omega),$ we can assume that $\left\Vert v\right\Vert
_{\infty}=1.$ Thus,
\[
N(\omega_{N})^{\frac{p}{N}}\left(  \frac{p-N}{p-1}\right)  ^{p-1}\left\vert
\Omega\right\vert ^{1-\frac{p}{N}}=\left\Vert \nabla v\right\Vert _{p}^{p}.
\]
But,
\[
N(\omega_{N})^{\frac{p}{N}}\left(  \frac{p-N}{p-1}\right)  ^{p-1}\left\vert
\Omega\right\vert ^{1-\frac{p}{N}}=\frac{N\omega_{N}}{(R^{\ast})^{p-N}}\left(
\frac{p-N}{p-1}\right)  ^{p-1}=\Lambda_{p}(B_{R^{\ast}})
\]
where, as before, $R^{\ast}=\left(  \left\vert \Omega\right\vert /\omega
_{N}\right)  ^{\frac{1}{N}}$ is such that $\left\vert B_{R^{\ast}}\right\vert
=\left\vert \Omega\right\vert .$ It follows that $\Lambda_{p}(B_{R^{\ast}%
})=\left\Vert \nabla v\right\Vert _{p}^{p}.$

Let $v^{\ast}\in W_{0}^{1,p}(B_{R})$ denote the Schwarz symmetrization of $v.$
We have $\left\Vert v^{\ast}\right\Vert _{\infty}=\left\Vert v\right\Vert
_{\infty}=1$ and
\[
\Lambda_{p}(B_{R^{\ast}})\leq\left\Vert \nabla v^{\ast}\right\Vert _{p}%
^{p}\leq\left\Vert \nabla v\right\Vert _{p}^{p}=\Lambda_{p}(B_{R^{\ast}}),
\]
from which we conclude that $\left\Vert \nabla v^{\ast}\right\Vert
_{p}=\left\Vert \nabla v\right\Vert _{p}.$ This fact implies that $\Omega$ is
a ball, according to \cite[Lemma 3.2]{Brothers}.
\end{proof}

\begin{corollary}
One has%
\[
\lim_{p\rightarrow N^{-}}\frac{\Lambda_{p}(\Omega)}{\left\vert p-N\right\vert
^{p-1}}=\frac{N\omega_{N}}{(N-1)^{N-1}}=\lim_{p\rightarrow N^{+}}\frac
{\Lambda_{p}(\Omega)}{\left\vert p-N\right\vert ^{p-1}}.
\]
In particular, the function $p\in(1,\infty)\mapsto\Lambda_{p}(\Omega)$ is
continuous at $p=N.$
\end{corollary}

\begin{proof}
It follows from (\ref{GamaoleqN}), (\ref{SNp}) and (\ref{volB1}) that%
\begin{align*}
\lim_{p\rightarrow N^{-}}\frac{\Lambda_{p}(\Omega)}{\left\vert p-N\right\vert
^{p-1}}  &  =\lim_{p\rightarrow N^{-}}\frac{\pi^{\frac{p}{2}}N}{(p-1)^{p-1}%
}\left(  \frac{\Gamma(N/p)\Gamma(1+N-N/p)}{\Gamma(1+N/2)\Gamma(N)}\right)
^{\frac{p}{N}}\\
&  =\frac{\pi^{\frac{p}{2}}N}{(N-1)^{N-1}}\frac{1}{\Gamma(1+N/2)}%
=\frac{N\omega_{N}}{(N-1)^{N-1}}.
\end{align*}
Now, by using (\ref{lowupp}) we obtain%
\[
\frac{N\omega_{N}}{(N-1)^{N-1}}=\lim_{p\rightarrow N^{+}}\frac{N(\omega
_{N})^{\frac{p}{N}}}{(p-1)^{p-1}}\left\vert \Omega\right\vert ^{1-\frac{p}{N}%
}\leq\lim_{p\rightarrow N^{+}}\frac{\Lambda_{p}(\Omega)}{\left\vert
p-N\right\vert ^{p-1}}%
\]
and%
\[
\lim_{p\rightarrow N^{+}}\frac{\Lambda_{p}(\Omega)}{\left\vert p-N\right\vert
^{p-1}}\leq\lim_{p\rightarrow N^{+}}\frac{N(\omega_{N})^{\frac{p}{N}}%
}{(p-1)^{p-1}}\left\vert B_{R_{\Omega}}\right\vert ^{1-\frac{p}{N}}%
=\frac{N\omega_{N}}{(N-1)^{N-1}}.
\]

The continuity follows, since
\[
\lim_{p\rightarrow N}\Lambda_{p}(\Omega)=\lim_{p\rightarrow N}\left\vert
p-N\right\vert ^{p-1}\lim_{p\rightarrow N}\frac{\Lambda_{p}(\Omega
)}{\left\vert p-N\right\vert ^{p-1}}=0=\Lambda_{N}(\Omega).
\]

\end{proof}

Theorem \ref{teomain2} says that the function $u_{p}(\left\vert x\right\vert
)$ defined in (\ref{wfunc}) is a positive extremal function of $\Lambda
_{p}(B_{R}).$ Let us prove that it is the unique.

\begin{theorem}
\label{propball}Let $R>0.$ The function $u_{p}(\left\vert x\right\vert )$
defined in (\ref{wfunc}) is the unique positive extremal function of
$\Lambda_{p}(B_{R}).$
\end{theorem}

\begin{proof}
It follows from Theorem \ref{teomain3} that
\[
-\Delta_{p}u_{p}=\Lambda_{p}(B_{R})\delta_{0}.
\]

Now, let us suppose that $v\in W_{0}^{1,p}(B_{R})$ is an arbitrary, positive
extremal function of $\Lambda_{p}(B_{R}).$ Let $v^{\ast}\in W_{0}^{1,p}%
(B_{R})$ denote the Schwarz symmetrization of $v$ (see \cite{Kawohl}). It
follows that $v^{\ast}$ is radial and radially nonincreasing and, moreover, it
satisfies $\left\Vert v^{\ast}\right\Vert _{\infty}=\left\Vert v\right\Vert
_{\infty}$ and $\left\Vert \nabla v^{\ast}\right\Vert _{p}^{p}\leq\left\Vert
\nabla v\right\Vert _{p}^{p}.$ Therefore, $v^{\ast}(\left\vert 0\right\vert
)=\left\Vert v^{\ast}\right\Vert _{\infty}=\left\Vert v\right\Vert _{\infty
}=1$ and%
\[
\Lambda_{p}(B_{R})\leq\left\Vert \nabla v^{\ast}\right\Vert _{p}^{p}%
\leq\left\Vert \nabla v\right\Vert _{p}^{p}=\Lambda_{p}(B_{R}).
\]
Thus, $v^{\ast}$ is also a nonnegative extremal function of $\Lambda_{p}%
(B_{R}).$ Theorem \ref{teomain3} yields $-\Delta_{p}v^{\ast}=\Lambda_{p}%
(B_{R})\delta_{0}=-\Delta_{p}u_{p},$ which implies that $v^{\ast}=u_{p}.$
Since
\[
\left\vert \nabla v^{\ast}(x)\right\vert =\left\vert \nabla u_{p}(\left\vert
x\right\vert )\right\vert =\frac{p-N}{p-1}R^{-\frac{p-N}{p-1}}\left\vert
x\right\vert ^{-\frac{N-1}{p-1}}>0,\text{ }0<\left\vert x\right\vert \leq R
\]
the set $\left\{  x\in B_{R}:\nabla v^{\ast}=0\right\}  $ has Lebesgue measure
zero. Hence, we can apply a well-known result (see \cite[Theorem
1.1]{Brothers}) to conclude that $v=v^{\ast}$ $(=u_{p}).$
\end{proof}

\begin{corollary}
Let $w_{q}$ denote the extremal function of $\lambda_{q}(B_{R}).$We have%
\begin{equation}
\lim_{q\rightarrow\infty}w_{q}(\left\vert x\right\vert )=1-\left(  \left\vert
x\right\vert /R\right)  ^{\frac{p-N}{p-1}}, \label{convball}%
\end{equation}
strongly in $C(\overline{B_{R}})$ and also in $W_{0}^{1,p}(B_{R}).$ Moreover,
(\ref{convball}) holds in $C^{1}(\overline{B_{\epsilon,R}})$ for each
$\epsilon\in(0,R),$ where $B_{\epsilon,R}:=\left\{  \epsilon<\left\vert
x\right\vert <R\right\}  .$
\end{corollary}

\begin{proof}
It follows from Theorem \ref{propball} that $1-\left(  \left\vert x\right\vert
/R\right)  ^{\frac{p-N}{p-1}}$ is the only limit function of the the family
$\left\{  w_{q}(\left\vert \cdot\right\vert )\right\}  ,$ as $q\rightarrow
\infty.$ Therefore, the convergence given by Theorem \ref{teomain4} is valid
for any sequence $q_{n}\rightarrow\infty$ and this guarantees that
(\ref{convball}) happens strongly in $C(\overline{B_{R}})$ and also in
$W_{0}^{1,p}(B_{R}).$

The convergence in $C^{1}(\overline{B_{\epsilon,R}})$ is consequence of the
following fact
\[
\lim_{q\rightarrow\infty}\lambda_{q}w_{q}(\left\vert x\right\vert
)^{q-1}=0,\text{ uniformly in }\overline{B_{\epsilon,R}},
\]
which occurs because of the uniform convergence of $w_{q}(\left\vert
x\right\vert )$ to $1-\left(  \left\vert x\right\vert /R\right)  ^{\frac
{p-N}{p-1}}.$ (Note that $0\leq w_{q}(\left\vert x\right\vert )\leq k<1$ for
some $k,$ and for all $x\in\overline{B_{\epsilon,R}}$ and all $q$ large
enough.) Therefore, we can apply a result of Lieberman (see \cite[Theorem
1]{Lieberman}) to guarantee that, for all $q$ large enough, $w_{q}$ is
uniformly bounded in the H\"{o}lder space $C^{1,\alpha}(\overline
{B_{\epsilon,R}}),$ for some $\alpha\in(0,1)$ that does not depend on $q.$
Then, we obtain the convergence (\ref{convball}) from the compactness of the
immersion $C^{1,\alpha}(\overline{B_{\epsilon,R}})\hookrightarrow
C^{1}(\overline{B_{\epsilon,R}})$ by taking into account that the limit
function is always $1-\left(  \left\vert x\right\vert /R\right)  ^{\frac
{p-N}{p-1}}.$
\end{proof}

\section{Asymptotics as $p\rightarrow\infty$\label{part2}}

In this section, $u_{p}\in W_{0}^{1,p}(\Omega)\cap C^{0,1-\frac{N}{p}%
}(\overline{\Omega})$ denotes a positive extremal function of $\Lambda
_{p}(\Omega)$ and $\rho\in W_{0}^{1,\infty}(\Omega)$ denotes the distance
function to the boundary $\partial\Omega.$ Thus, $0<u_{p}(x)\leq\left\Vert
u_{p}\right\Vert _{\infty}=1$ for all $x\in\Omega,$
\begin{equation}
\Lambda_{p}(\Omega)=\min\left\{  \left\Vert \nabla u\right\Vert _{p}^{p}:u\in
W_{0}^{1,p}(\Omega)\text{ and }\left\Vert u\right\Vert _{\infty}=1\right\}
=\left\Vert \nabla u_{p}\right\Vert _{p}^{p} \label{Lamb}%
\end{equation}
and%
\[
\rho(x)=\inf_{y\in\partial\Omega}\left\vert y-x\right\vert ,\text{ \ }%
x\in\overline{\Omega}.
\]

As shown in Section \ref{part1}, $u_{p}$ has a unique maximum point, denoted
by $x_{p},$ and
\[
\left\{
\begin{array}
[c]{ll}%
-\Delta_{p}u_{p}=\Lambda_{p}(\Omega)\delta_{x_{p}} & \text{in }\Omega\\
u_{p}=0 & \text{on }\partial\Omega.
\end{array}
\right.
\]

It is convenient to recall some properties of the distance function:

\begin{enumerate}
\item[(P1)] $\rho\in W_{0}^{1,r}(\Omega)\ $\ for all $1\leq r\leq\infty,$

\item[(P2)] $\left\vert \nabla\rho\right\vert =1$ almost everywhere in
$\Omega,$

\item[(P3)] $\left\Vert \rho\right\Vert _{\infty}=R_{\Omega}$ is the radius of
the largest ball contained in $\Omega,$

\item[(P4)] $\dfrac{1}{\left\Vert \rho\right\Vert _{\infty}}\leq
\dfrac{\left\Vert \nabla\phi\right\Vert _{\infty}}{\left\Vert \phi\right\Vert
_{\infty}}$ for all $0\not \equiv \phi\in W_{0}^{1,\infty}(\Omega).$
\end{enumerate}

Let us, for a moment, consider $\Omega=B_{R}.$ For this domain
\[
\rho(x)=R-\left\vert x\right\vert ;\text{ \ }0\leq\left\vert x\right\vert \leq
R
\]
and, accordingly to (\ref{p>N}) and (\ref{wfunc}): $x_{p}=0$ for all $p>N$,%
\begin{equation}
\lim_{p\rightarrow\infty}\Lambda_{p}(B_{R})^{\frac{1}{p}}=\lim_{p\rightarrow
\infty}\left(  \frac{N\omega_{N}}{R^{p-N}}\right)  ^{\frac{1}{p}}\left(
\frac{p-N}{p-1}\right)  ^{\frac{p-1}{p}}=\frac{1}{R}=\frac{1}{\left\Vert
\rho\right\Vert _{\infty}} \label{Lamball}%
\end{equation}
and
\begin{equation}
\lim_{p\rightarrow\infty}u_{p}(x)=\lim_{p\rightarrow\infty}1-\left(
\frac{\left\vert x\right\vert }{R}\right)  ^{\frac{p-N}{p-1}}=1-\frac
{\left\vert x\right\vert }{R}=\frac{\rho(x)}{\left\Vert \rho\right\Vert
_{\infty}};\text{ \ }0\leq\left\vert x\right\vert \leq R. \label{wpball}%
\end{equation}

As we will see in the sequel, (\ref{Lamball}) holds for any bounded domain,
whereas (\ref{wpball}) holds only for some special domains.

\begin{lemma}
\label{monotG}The function $p\in(N,\infty)\longmapsto\Lambda_{p}%
(\Omega)^{\frac{1}{p}}\left\vert \Omega\right\vert ^{-\frac{1}{p}}$ is increasing.
\end{lemma}

\begin{proof}
Let $N<p_{1}<p_{2}$ and, for each $i\in\left\{  1,2\right\}  $ let $u_{p_{i}%
}\in W_{0}^{1,p_{i}}(\Omega)$ denote a positive extremal function of
$\Lambda_{p_{i}}(\Omega).$ H\"{o}lder inequality implies that%
\[
\Lambda_{p_{1}}(\Omega)\leq\int_{\Omega}\left\vert \nabla u_{p_{2}}\right\vert
^{p_{1}}dx\leq\left(  \int_{\Omega}\left\vert \nabla u_{p_{2}}\right\vert
^{p_{2}}dx\right)  ^{\frac{p_{1}}{p_{2}}}\left\vert \Omega\right\vert
^{1-\frac{p_{1}}{p_{2}}}=\Lambda_{p_{2}}(\Omega)^{\frac{p_{1}}{p_{2}}%
}\left\vert \Omega\right\vert ^{1-\frac{p_{1}}{p_{2}}},
\]
so that
\[
\Lambda_{p_{1}}(\Omega)^{\frac{1}{p_{1}}}\left\vert \Omega\right\vert
^{-\frac{1}{p_{1}}}\leq\Lambda_{p_{2}}(\Omega)^{\frac{1}{p_{2}}}\left\vert
\Omega\right\vert ^{-\frac{1}{p_{2}}.}%
\]

\end{proof}

An immediate consequence of this lemma is that the function $p\in
(N,\infty)\longmapsto\Lambda_{p}(\Omega)$ is increasing.

\begin{theorem}
\label{increas}One has
\[
\lim_{p\rightarrow\infty}\Lambda_{p}(\Omega)^{\frac{1}{p}}=\lim_{p\rightarrow
\infty}\Lambda_{p}(\Omega)^{\frac{1}{p}}\left\vert \Omega\right\vert
^{-\frac{1}{p}}=\frac{1}{\left\Vert \rho\right\Vert _{\infty}}.
\]

\end{theorem}

\begin{proof}
It is enough to prove that
\[
\lim_{p\rightarrow\infty}\Lambda_{p}(\Omega)^{\frac{1}{p}}\left\vert
\Omega\right\vert ^{-\frac{1}{p}}=\frac{1}{\left\Vert \rho\right\Vert
_{\infty}}.
\]

It follows from (\ref{Lamb}) that
\[
\Lambda_{p}(\Omega)^{\frac{1}{p}}\left\vert \Omega\right\vert ^{-\frac{1}{p}%
}\leq\frac{\left\Vert \nabla\rho\right\Vert _{p}}{\left\Vert \rho\right\Vert
_{\infty}}\left\vert \Omega\right\vert ^{-\frac{1}{p}}=\frac{1}{\left\Vert
\rho\right\Vert _{\infty}},\text{ \ }p>N.
\]
Hence, the monotonicity proved in Lemma \ref{monotG} guarantees that%
\[
\Lambda_{p}(\Omega)^{\frac{1}{p}}\left\vert \Omega\right\vert ^{-\frac{1}{p}%
}\leq L:=\lim_{s\rightarrow\infty}\Lambda_{s}(\Omega)^{\frac{1}{s}}\left\vert
\Omega\right\vert ^{-\frac{1}{s}}=\lim_{s\rightarrow\infty}\Lambda_{s}%
(\Omega)^{\frac{1}{s}}\leq\frac{1}{\left\Vert \rho\right\Vert _{\infty}%
},\text{ \ for all }p>N.
\]

We are going to show that $L=\frac{1}{\left\Vert \rho\right\Vert _{\infty}}.$
For this, let us fix $r>N.$ Since%
\[
\left\Vert \nabla u_{p}\right\Vert _{r}\leq\left\Vert \nabla u_{p}\right\Vert
_{p}\left\vert \Omega\right\vert ^{\frac{1}{r}-\frac{1}{p}}=\Lambda_{p}%
(\Omega)^{\frac{1}{p}}\left\vert \Omega\right\vert ^{-\frac{1}{p}}\left\vert
\Omega\right\vert ^{\frac{1}{r}}\leq L\left\vert \Omega\right\vert ^{\frac
{1}{r}},\text{ \ }p>r
\]
the family $\left\{  u_{p}\right\}  _{p>r}$ is uniformly bounded in
$W_{0}^{1,r}(\Omega).$ It follows that there exist $p_{n}\rightarrow\infty$
and $u_{\infty}\in W_{0}^{1,r}(\Omega)$ such that%
\[
u_{p_{n}}\rightharpoonup u_{\infty}\text{ (weakly) in }W_{0}^{1,r}(\Omega).
\]
Thus,%
\[
\left\Vert \nabla u_{\infty}\right\Vert _{r}\leq\liminf_{n}\left\Vert \nabla
u_{p_{n}}\right\Vert _{r}\leq L\left\vert \Omega\right\vert ^{\frac{1}{r}}.
\]

After passing to another subsequence, if necessary, the compactness of the
immersion $W_{0}^{1,r}(\Omega)\hookrightarrow C(\overline{\Omega})$ yields
\[
u_{p_{n}}\rightarrow u_{\infty}\text{ (strongly) in }C(\overline{\Omega}).
\]
Note that $\left\Vert u_{\infty}\right\Vert _{\infty}=1$ since $\left\Vert
u_{p}\right\Vert _{\infty}=1$ for all $p>N.$

The uniform convergence $u_{p_{n}}\rightarrow u_{\infty}$ implies that, if
$s>r,$ then $u_{\infty}$ is also the weak limit in $W_{0}^{1,s}(\Omega)$ of a
subsequence of $\left\{  u_{p_{n}}\right\}  .$ Therefore,
\[
u_{\infty}\in W_{0}^{1,s}(\Omega)\text{ \ \ and \ \ }\left\Vert \nabla
u_{\infty}\right\Vert _{s}\leq L\left\vert \Omega\right\vert ^{\frac{1}{s}%
},\text{ \ for all }s>r,
\]
implying that $u_{\infty}\in W_{0}^{1,\infty}(\Omega)$ and
\[
\left\Vert \nabla u_{\infty}\right\Vert _{\infty}\leq L\leq\frac{1}{\left\Vert
\rho\right\Vert _{\infty}}.
\]

Combining this fact with Property P4 (recall that $\left\Vert u_{\infty
}\right\Vert _{\infty}=1$) we conclude that
\[
\left\Vert \nabla u_{\infty}\right\Vert _{\infty}\leq L\leq\frac{1}{\left\Vert
\rho\right\Vert _{\infty}}\leq\left\Vert \nabla u_{\infty}\right\Vert
_{\infty},
\]
from which we obtain%
\[
L=\frac{1}{\left\Vert \rho\right\Vert _{\infty}}=\left\Vert \nabla u_{\infty
}\right\Vert _{\infty}.
\]

\end{proof}

It is interesting to notice that $\Lambda_{p}(\Omega)^{\frac{1}{p}}$ and
$\lambda_{p}(\Omega)^{\frac{1}{p}}$ have the same asymptotic behavior as
$p\rightarrow\infty,$ since
\[
\lim_{p\rightarrow\infty}\lambda_{p}(\Omega)^{\frac{1}{p}}=\frac{1}{\left\Vert
\rho\right\Vert _{\infty}},
\]
as proved in \cite{Fuka, ARMA99}, where the infinity-eigenvalue problem was
studied as the limit problem of the standard eigenvalue problem for the
$p$-Laplacian, as $p\rightarrow\infty.$

\begin{theorem}
\label{uinf}There exist $p_{n}\rightarrow\infty,$ $x_{\ast}\in\Omega$ and
$u_{\infty}\in W_{0}^{1,\infty}(\Omega)$ such that:
\end{theorem}

\begin{enumerate}
\item[(i)] $u_{p_{n}}$ converges to $u_{\infty}$ weakly in $W_{0}^{1,r}%
(\Omega),$ for any $r>N,$ and uniformly in $\overline{\Omega};$

\item[(ii)] $\left\Vert \nabla u_{\infty}\right\Vert _{\infty}=\dfrac
{1}{\left\Vert \rho\right\Vert _{\infty}};$

\item[(iii)] $0\leq u_{\infty}\leq\dfrac{\rho}{\left\Vert \rho\right\Vert
_{\infty}}$ a.e. in $\Omega;$

\item[(iv)] $x_{p_{n}}\rightarrow x_{\ast};$

\item[(v)] $u_{\infty}(x_{\ast})=1=\left\Vert u_{\infty}\right\Vert _{\infty}$
and $\rho(x_{\ast})=\left\Vert \rho\right\Vert _{\infty}.$
\end{enumerate}

\begin{proof}
Items (i) and (ii) follow from the proof of the previous theorem. In
particular, (ii) says that the Lipschitz constant of $\left\Vert
\rho\right\Vert _{\infty}u_{\infty}$ is $\left\Vert \nabla(\left\Vert
\rho\right\Vert _{\infty}u_{\infty})\right\Vert _{\infty}=1.$ Thus,
\[
0<\left\Vert \rho\right\Vert _{\infty}u_{\infty}(x)\leq\left\vert
x-y\right\vert ,\text{ for almost all }x\in\Omega\text{ and }y\in
\partial\Omega
\]
and hence we obtain $\left\Vert \rho\right\Vert _{\infty}u_{\infty}\leq\rho$
a.e. in $\Omega,$ as affirmed in (iii). Of course, $\left\{  p_{n}\right\}  $
can be chosen such that $x_{p_{n}}\rightarrow x_{\ast}$ for some $x_{\ast}%
\in\Omega,$ yielding (iv). Since $u_{p_{n}}(x_{p_{n}})=1,$ the uniform
convergence $u_{p_{n}}\rightarrow u_{\infty}$ implies that $u_{\infty}%
(x_{\ast})=1.$ Therefore, (iii) implies that $\left\Vert \rho\right\Vert
_{\infty}=\rho(x_{\ast}),$ what concludes the proof of (v).
\end{proof}

\begin{remark}
We will prove in the sequel that $x_{\ast}$ is the only maximum point of
$u_{\infty}$ and that $u_{\infty}$ is infinity harmonic in the punctured
domain $\Omega\backslash\left\{  x_{\ast}\right\}  .$
\end{remark}

\begin{remark}
Item (ii) of Theorem \ref{uinf} and property (P4) above imply that $u_{\infty
}$ minimizes the Rayleigh quotient $\dfrac{\left\Vert \nabla u\right\Vert
_{\infty}}{\left\Vert u\right\Vert _{\infty}}$ among all nontrivial functions
$u$ in $W_{0}^{1,\infty}(\Omega).$ This property is also shared with the
distance function $\rho$ and the first eigenfunctions of the $\infty
$-Laplacian (see \cite{ARMA99}). In the sequel (see Theorem \ref{u=dist}) we
will prove that $u_{\infty}=\frac{\rho}{\left\Vert \rho\right\Vert _{\infty}}$
for some special domains. For such domains $u_{\infty}$ is also a first
eigenfunction of the $\infty$-Laplacian, according to \cite[Theorem 2.7]{Yu}.
\end{remark}

In order to gain some insight on which equation $u_{\infty}$ satisfies, let us
go back to the case $\Omega=B_{R}.$ It follows from (\ref{wpball}) that:
\[
u_{\infty}=\frac{\rho}{\left\Vert \rho\right\Vert _{\infty}}=1-\frac
{\left\vert x\right\vert }{R},
\]
$x_{\ast}=0$ and $u_{\infty}(0)=1=\frac{\rho(0)}{\left\Vert \rho\right\Vert
_{\infty}}.$ Moreover, it is easy to check that $u_{\infty}\in C(\overline
{B_{R}})\cap C^{2}(B_{R}\backslash\left\{  0\right\}  ),$ $\nabla u_{\infty
}\not =0$ in $B_{R}\backslash\left\{  0\right\}  $ and
\[
\Delta_{\infty}u_{\infty}(x)=0,\text{ }x\in B_{R}\backslash\left\{  0\right\}
,
\]
where $\Delta_{\infty}$ denotes the $\infty$-Laplacian (see \cite{Aron, BDM,
Cr, CEG, Lq}), defined by%
\[
\Delta_{\infty}\phi:=\frac{1}{2}\left\langle \nabla\phi,\nabla\left\vert
\nabla\phi\right\vert ^{2}\right\rangle =\sum_{i,j=1}^{N}\frac{\partial\phi
}{\partial x_{i}}\frac{\partial\phi}{\partial x_{j}}\frac{\partial^{2}\phi
}{\partial x_{i}\partial x_{j}}.
\]

After this motivation, let us to show that the function $u_{\infty}$ given by
Theorem \ref{uinf} is $\infty$-harmonic in $\Omega\backslash\left\{  x_{\ast
}\right\}  ,$ i.e. that it satisfies $\Delta_{\infty}u=0$ in $\Omega
\backslash\left\{  x_{\ast}\right\}  $ in the viscosity sense. First, we need
to recall some definitions regarding the viscosity approach for the equation
$\Delta_{p}u=0,$ where $N<p\leq\infty.$

\begin{definition}
\label{def3}Let $u\in C(\overline{\Omega}),$ $x_{0}\in\Omega$ and $\phi\in
C^{2}(\Omega).$ We say that $\phi$ touches $u$ at $x_{0}$ from below if%
\[
\phi(x)-u(x)<0=\phi(x_{0})-u(x_{0}),\text{ \ for all }x\in\Omega
\backslash\{x_{0}\}.
\]
Analogously, we say that $\phi$ touches $u$ at $x_{0}$ from above if%
\[
\phi(x)-u(x)>0=\phi(x_{0})-u(x_{0}),\text{ \ for all }x\in\Omega
\backslash\{x_{0}\}.
\]

\end{definition}

\begin{definition}
\label{def4}Let $N<p\leq\infty$ and $u\in C(\overline{\Omega}).$ We say that
$u$ is $p$-subharmonic in $\Omega$ in the viscosity sense, if
\[
\Delta_{p}\phi(x_{0})\geq0
\]
whenever $x_{0}\in\Omega$ and $\phi\in C^{2}(\Omega)$ are such that $\phi$
touches $u$ from above at $x_{0}.$ Analogously, we say that $u$ is
$p$-superharmonic in $\Omega$ in the viscosity sense, if
\[
\Delta_{p}\phi(x_{0})\leq0
\]
whenever $x_{0}\in\Omega$ and $\phi\in C^{2}(\Omega)$ are such that $\phi$
touches $u$ from below at $x_{0}.$
\end{definition}

\begin{definition}
Let $N<p\leq\infty$ and $u\in C(\overline{\Omega}).$ We say that $u$ is
$p$-harmonic in $\Omega,$ in the viscosity sense, if $u$ is both:
$p$-subharmonic and $p$-superharmonic in $\Omega,$ in the viscosity sense. We
write $\Delta_{\infty}u=0$ in $\Omega$ to mean that $u$ is $\infty$-harmonic
in $\Omega,$ in the viscosity sense.
\end{definition}

In Definitions \ref{def3} and \ref{def4}, we mean
\[
\Delta_{p}\phi(x_{0}):=\left\vert \nabla\phi(x_{0})\right\vert ^{p-4}\left\{
\left\vert \nabla\phi(x_{0})\right\vert ^{2}\Delta\phi(x_{0})+(p-2)\Delta
_{\infty}\phi(x_{0})\right\}  ,\text{ \ }N<p<\infty
\]
and%
\[
\Delta_{\infty}\phi(x_{0}):=\sum_{i,j=1}^{N}\frac{\partial\phi}{\partial
x_{i}}(x_{0})\frac{\partial\phi}{\partial x_{j}}(x_{0})\frac{\partial^{2}\phi
}{\partial x_{i}\partial x_{j}}(x_{0}).
\]

The following two Lemmas can be found in \cite{Lq}.

\begin{lemma}
Suppose $u\in C(\Omega)\cap W^{1,p}(\Omega)$ satisfies $\Delta_{p}u\geq0$
(resp. $\Delta_{p}u\leq0$) in $\Omega,$ in the weak sense, then $u$ is
$p$-subharmonic (resp. $p$-superharmonic) in $\Omega,$ in the viscosity sense.
\end{lemma}

\begin{lemma}
\label{calc}Suppose that $f_{n}\rightarrow f$ uniformly in $\overline{\Omega
},$ $f_{n},\;f\in C(\overline{\Omega}).$ If $\phi\in C^{2}(\Omega)$ touches
$f$ from below at $y_{0},$ then there exists $y_{n_{j}}\rightarrow y_{0}$ such
that
\[
f(y_{n_{j}})-\phi(y_{n_{j}})=\min_{\Omega}\left\{  f_{n_{j}}-\phi\right\}  .
\]

\end{lemma}

From now on, $u_{\infty}$ and $x_{\ast}$ are as in Theorem \ref{uinf}.

\begin{theorem}
\label{diric}The function $u_{\infty}$ satisfies%
\begin{equation}
\left\{
\begin{array}
[c]{ll}%
\Delta_{\infty}v=0 & \text{in }\Omega\backslash\left\{  x_{\ast}\right\} \\
v=\frac{\rho}{\left\Vert \rho\right\Vert _{\infty}} & \text{on }\left\{
x_{\ast}\right\}  \cup\partial\Omega
\end{array}
\right.  \label{infdiric}%
\end{equation}
in the viscosity sense.
\end{theorem}

\begin{proof}
Since $u_{\infty}=\frac{\rho}{\left\Vert \rho\right\Vert _{\infty}}$ on
$\left\{  x_{\ast}\right\}  \cup\partial\Omega,$ it remains to check that
$\Delta_{\infty}u_{\infty}=0$ in $\Omega\backslash\left\{  x_{\ast}\right\}
.$ Let $\xi\in\Omega\backslash\left\{  x_{\ast}\right\}  $ and take $\phi\in
C^{2}(\Omega\backslash\left\{  x_{\ast}\right\}  )$ touching $u_{\infty}$ from
below at $\xi.$ Thus,
\[
\phi(x)-u_{\infty}(x)<0=\phi(\xi)-u_{\infty}(\xi),\text{ if }x\not =\xi.
\]

If $\left\vert \nabla\phi(\xi)\right\vert =0$ then we readily obtain
\[
\Delta_{\infty}\phi(\xi)=\sum_{i,j=1}^{N}\frac{\partial\phi}{\partial x_{i}%
}(\xi)\frac{\partial\phi}{\partial x_{j}}(\xi)\frac{\partial^{2}\phi}{\partial
x_{i}\partial x_{j}}(\xi)=0.
\]

Otherwise, if $\left\vert \nabla\phi(\xi)\right\vert \not =0$ let us take a
ball $B_{\epsilon}(\xi)\subset\Omega\backslash\left\{  x_{\ast}\right\}  $
such that $\left\vert \nabla\phi\right\vert >0$ in $B_{\epsilon}(\xi).$ Let
$n_{0}>N$ be such that $x_{p_{n}}\not \in B_{\epsilon}(\xi)$ for all
$n>n_{0}.$ This is possible because $x_{p_{n}}\rightarrow x_{\ast}\not =\xi.$
It follows that $u_{p_{n}}$ is $p_{n}$-harmonic in $B_{\epsilon}(\xi)$ in the
viscosity sense.

According Lemma \ref{calc}, let $\left\{  \xi_{n_{j}}\right\}  \subset
B_{\epsilon}(\xi)$ such that $\xi_{n_{j}}\rightarrow\xi$ and
\[
m_{j}:=\min_{B_{\epsilon}(\xi)}\left\{  u_{p_{n_{j}}}-\phi\right\}
=u_{p_{n_{j}}}(\xi_{n_{j}})-\phi(\xi_{n_{j}})\leq u_{p_{n_{j}}}(x)-\phi
(x),\text{ \ }x\not =\xi_{n_{j}}.
\]
The function $\psi(x):=\phi(x)+m_{j}-\left\vert x-\xi_{n_{j}}\right\vert ^{4}$
belongs to $C^{2}(B_{\epsilon}(\xi))$ and touches $u_{n_{j}}$ from below at
$\xi_{n_{j}}.$ Indeed,
\begin{align*}
\psi(x)-u_{p_{n_{j}}}(x)  &  =\phi(x)-u_{p_{n_{j}}}(x)+m_{j}-\left\vert
x-\xi_{n_{j}}\right\vert ^{4}\\
&  \leq-\left\vert x-\xi_{n_{j}}\right\vert ^{4}<0=\psi(\xi_{n_{j}%
})-u_{p_{n_{j}}}(\xi_{n_{j}}),\text{ \ }x\not =\xi_{n_{j}}.
\end{align*}
Thus, $\Delta_{p_{n_{j}}}\psi(\xi_{n_{j}})\leq0,$ since $u_{p_{n_{j}}}$ is
$p_{n_{j}}$-harmonic in $B_{\epsilon}(\xi).$ Hence,%
\[
0\geq\Delta_{p_{n_{j}}}\psi(\xi_{n_{j}})=\left\vert \nabla\psi(\xi_{n_{j}%
})\right\vert ^{p_{n_{j}}-4}\left\{  \left\vert \nabla\psi(\xi_{n_{j}%
})\right\vert ^{2}\Delta\psi(\xi_{n_{j}})+(p_{n_{j}}-2)\Delta_{\infty}\psi
(\xi_{n_{j}})\right\}
\]
from which we obtain%
\[
\Delta_{\infty}\phi(\xi_{n_{j}})=\Delta_{\infty}\psi(\xi_{n_{j}})\leq
-\frac{\left\vert \nabla\psi(\xi_{n_{j}})\right\vert ^{2}}{p_{n_{j}}-2}%
\Delta\psi(\xi_{n_{j}}).
\]
So, by making $j\rightarrow\infty$ we conclude that $\Delta_{\infty}\phi
(\xi)\leq0.$

We have proved that $u_{\infty}$ is $\infty$-superharmonic in $\Omega
\backslash\left\{  x_{\ast}\right\}  ,$ in the viscosity sense. Analogously,
we can prove that $u_{\infty}$ is also $\infty$-subharmonic in $\Omega
\backslash\left\{  x_{\ast}\right\}  ,$ in the viscosity sense.
\end{proof}

We recall that $u_{\infty}$ is the only solution of the Dirichlet problem
(\ref{infdiric}). This uniqueness result is a consequence of the following
comparison principle (see \cite{BB, Jensen}):

\begin{theorem}
[Comparison Principle]Let $D$ be a bounded domain and let $u,v\in
C(\overline{D})$ satisfying $\Delta_{\infty}u\geq0$ in $D$ and $\Delta
_{\infty}v\leq0$ in $D.$ If $u\leq v$ on $\partial D,$ then $u\leq v$ in $D.$
\end{theorem}

\begin{theorem}
\label{unic}The function $u_{\infty}$ is strictly positive in $\Omega$ and
attains its maximum value $1$ only at $x_{\ast}.$
\end{theorem}

\begin{proof}
Let $D:=\Omega\backslash\left\{  x_{\ast}\right\}  .$ Since $u_{\infty
}(x_{\ast})>0$ and $u_{\infty}$ is nonnegative and $\infty$-harmonic in $D,$
it follows from the Harnack inequality for the infinity harmonic functions
(see \cite{LqManf}) that $Z_{\infty}:=\left\{  x\in\Omega:u_{\infty
}(x)=0\right\}  $ is open in $\Omega.$ Since $Z_{\infty}$ is also closed and
$Z_{\infty}\not =\Omega,$ we conclude that $Z_{\infty}$ is empty, so that
$u>0$ in $\Omega.$

Let $m:=\max\left\{  \left\vert x-x_{\ast}\right\vert :x\in\partial
\Omega\right\}  $ and $v(x):=1-\frac{1}{m}\left\vert x-x_{\ast}\right\vert ,$
$x\in\Omega.$ It is easy to check that $\Delta_{\infty}v=0$ in $D$ and that
$v\geq u_{\infty}$ on $\partial D=\left\{  x_{\ast}\right\}  \cup
\partial\Omega.$ Therefore, by the comparison principle above, we have
\[
u_{\infty}(x)\leq v(x)=1-\frac{1}{m}\left\vert x-x_{\ast}\right\vert
<1=\left\Vert u_{\infty}\right\Vert _{\infty},\text{ \ for all }x\in
\Omega\backslash\left\{  x_{\ast}\right\}  .
\]

\end{proof}

Since $x_{\ast}$ is also a maximum point of the distance function $\rho$, an
immediate consequence of the previous theorem is that if $\Omega$ is such that
$\rho$ has a unique maximum point, then the family $\left\{  u_{p}\right\}
_{p>N}$ converges, as $p\rightarrow\infty,$ to the unique solution $u_{\infty
}$ of the Dirichlet problem (\ref{infdiric}). However, this property of
$\Omega$ alone does not assure that $u_{\infty}=\frac{\rho}{\left\Vert
\rho\right\Vert _{\infty}}.$ For example, for the square $S=\left\{
(x,y)\in\mathbb{R}^{2}:\left\vert x\right\vert +\left\vert y\right\vert
<1\right\}  $ the origin is the unique maximum point of the distance function
$\rho,$ but one can check from \cite[Proposition 4.1]{ARMA99} that $\rho$ is
not $\infty$-harmonic at the points of $\Omega$ on the coordinate axes. As a
matter of fact, for a general bounded domain $\Omega$ the distance function
fails to be $\infty$-harmonic exactly on the \textit{ridge of }$\Omega,$ the
set $\mathcal{R}(\Omega)$ of all points in $\Omega$ whose distance to the
boundary is reached at least at two points in $\partial\Omega.$ This
well-known fact can be proved by combining Corollaries 3.4 and 4.4 of
\cite{CEG}, as pointed out in \cite[Lemma 2.6]{Yu}. Note that $\mathcal{R}(S)$
is set of the points in $S$ that are on the coordinate axes. As we will see in
the sequel, the complementary condition to guarantee that $u_{\infty}%
=\frac{\rho}{\left\Vert \rho\right\Vert _{\infty}}$ is $\mathcal{R}%
(\Omega)=\left\{  x_{0}\right\}  ,$ where $x_{0}$ denotes the unique maximum
point of $\rho.$

\begin{theorem}
\label{u=dist}One has $u_{\infty}=\frac{\rho}{\left\Vert \rho\right\Vert
_{\infty}}$ in $\overline{\Omega}$ if, and only if:
\end{theorem}

\begin{enumerate}
\item[(i)] $\rho$ has a unique maximum point, say $x_{0},$ and

\item[(ii)] for each $x\in\Omega\backslash\left\{  x_{0}\right\}  $ there
exists a unique $y_{x}\in\partial\Omega$ such that $\left\vert x-y_{x}%
\right\vert =\rho(x).$
\end{enumerate}

\begin{proof}
If $u_{\infty}=\frac{\rho}{\left\Vert \rho\right\Vert _{\infty}}$ then
$x_{\ast}$ is the only maximum point of the distance function $\rho,$
according Theorems \ref{uinf} and \ref{unic}. It follows from Theorem
\ref{diric} that $\Delta_{\infty}\rho=0$ in $\Omega\backslash\left\{  x_{\ast
}\right\}  .$ Hence, $\mathcal{R}(\Omega)=\left\{  x_{0}\right\}  ,$ which is
equivalent to $\mathrm{(ii).}$

Conversely, item $\mathrm{(i)}$ and Theorem \ref{uinf} imply that
$x_{0}=x_{\ast},$ whereas item $\mathrm{(ii)}$ implies that $\mathcal{R}%
(\Omega)=\left\{  x_{0}\right\}  .$ It follows that $\frac{\rho}{\left\Vert
\rho\right\Vert _{\infty}}$ satisfies (\ref{infdiric}). Hence, uniqueness of
the viscosity solution of this Dirichlet problem guarantees that $u_{\infty
}=\frac{\rho}{\left\Vert \rho\right\Vert _{\infty}}.$
\end{proof}

Balls, ellipses and other convex domains satisfy conditions $\mathrm{(i)}$ and
$\mathrm{(ii)}$.

\subsection{Multiplicity of minimizers of the quotient $\frac{\left\Vert
\nabla\phi\right\Vert _{\infty}}{\left\Vert \phi\right\Vert _{\infty}}$ in
$W_{0}^{1,\infty}(\Omega)\backslash\left\{  0\right\}  .$}

In this subsection we show that each maximum point $x_{0}$ of the distance
function $\rho$ gives rise to a positive function $u\in W_{0}^{1,\infty
}(\Omega)\backslash\left\{  0\right\}  $ satisfying
\begin{equation}
\left\Vert u\right\Vert _{\infty}=1\text{ \ and \ }\left\Vert \nabla
u\right\Vert _{\infty}=\frac{1}{\left\Vert \rho\right\Vert _{\infty}}%
=\min\left\{  \frac{\left\Vert \nabla\phi\right\Vert _{\infty}}{\left\Vert
\phi\right\Vert _{\infty}}:\phi\in W_{0}^{1,\infty}(\Omega)\backslash\left\{
0\right\}  \right\}  . \label{quotient}%
\end{equation}
Moreover, such a function attains its maximum value only at $x_{0}.$ In
particular, we conclude that for an annulus, there exist infinitely many
positive and nonradial functions satisfying (\ref{quotient}).

\begin{proposition}
\label{distmaxext}Let $x_{0}\in\mathbb{R}^{N}$ and let $u_{\infty}\in
C(\overline{\Omega})$ be the unique viscosity solution of the following
Dirichlet problem
\begin{equation}
\left\{
\begin{array}
[c]{ll}%
\Delta_{\infty}u=0 & \text{in }\Omega\backslash\left\{  x_{0}\right\} \\
u=0 & \text{on }\partial\Omega\\
u(x_{0})=1. &
\end{array}
\right.  \label{infdiric1}%
\end{equation}
Then,
\end{proposition}

\begin{enumerate}
\item[(i)] $0<u_{\infty}(x)<1$ for all $x\in\Omega\backslash\left\{
x_{0}\right\}  .$

\item[(ii)] if $x_{0}$ is a maximum point of the distance function $\rho,$
then $\left\Vert u_{\infty}\right\Vert _{\infty}=1$ and
\begin{equation}
\left\Vert \nabla u_{\infty}\right\Vert _{\infty}=\frac{1}{\left\Vert
\rho\right\Vert _{\infty}}. \label{minf}%
\end{equation}

\end{enumerate}

\begin{proof}
Following the proof of Theorem \ref{unic}, we obtain item $\mathrm{(i)}$ by
combining Harnack inequality and comparison principle in $D:=\Omega
\backslash\left\{  x_{0}\right\}  .$

In order to prove $\mathrm{(ii)}$ we first show that
\[
u_{\infty}=\lim_{p\rightarrow\infty}u_{p},\text{ uniformly in }\overline
{\Omega}%
\]
where%
\[
\left\{
\begin{array}
[c]{ll}%
-\Delta_{p}u_{p}=\Lambda_{p}(\Omega)\delta_{x_{0}} & \text{in }\Omega\\
u=0 & \text{on }\partial\Omega.
\end{array}
\right.
\]

It is easy to check that $-\Delta_{p}u_{p}\geq0$ in $\Omega,$ in the weak
sense. Hence, according the weak comparison principle, $u_{p}\geq0$ in
$\Omega.$

Since
\[
\Lambda_{p}(\Omega)\left\Vert u_{p}\right\Vert _{\infty}^{p}\leq\left\Vert
\nabla u_{p}\right\Vert _{p}^{p}=\Lambda_{p}(\Omega)u_{p}(x_{0})\leq
\Lambda_{p}(\Omega)\left\Vert u_{p}\right\Vert _{\infty}%
\]
we conclude that%
\[
u_{p}(x_{0})\leq\left\Vert u_{p}\right\Vert _{\infty}\leq1\text{ \ \ and
\ \ }\left\Vert \nabla u_{p}\right\Vert _{p}\leq\Lambda_{p}(\Omega)^{\frac
{1}{p}}.
\]

Let $r\in(N,p).$ Since%
\[
\left\Vert \nabla u_{p}\right\Vert _{r}\leq\left\Vert \nabla u_{p}\right\Vert
_{p}\left\vert \Omega\right\vert ^{\frac{1}{r}-\frac{1}{p}}\leq\Lambda
_{p}(\Omega)^{\frac{1}{p}}\left\vert \Omega\right\vert ^{-\frac{1}{p}%
}\left\vert \Omega\right\vert ^{\frac{1}{r}}\leq\frac{\left\vert
\Omega\right\vert ^{\frac{1}{r}}}{\left\Vert \rho\right\Vert _{\infty}},\text{
\ }p>r
\]
the family $\left\{  u_{p}\right\}  _{p>r}$ is uniformly bounded in
$W_{0}^{1,r}(\Omega).$ It follows, as in the proof of Proposition
\ref{increas}, that there exist $p_{n}\rightarrow\infty$ and $U_{\infty}\in
W_{0}^{1,\infty}(\Omega)$ such that $u_{p_{n}}\rightarrow U_{\infty}$
(strongly) in $C(\overline{\Omega})$ with
\begin{equation}
\left\Vert \nabla U_{\infty}\right\Vert _{\infty}\leq\dfrac{1}{\left\Vert
\rho\right\Vert _{\infty}}\text{ \ \ and \ \ }U_{\infty}\leq\frac{\rho
}{\left\Vert \rho\right\Vert _{\infty}}\text{ a.e. in }\Omega.\label{aux20}%
\end{equation}

Now are going to show that $U_{\infty}(x_{0})=\left\Vert U_{\infty}\right\Vert
_{\infty}=1.$

Since
\begin{align*}
\Lambda_{p}(\Omega)\rho(x_{0}) &  =\int_{\Omega}\left\vert \nabla
u_{p}\right\vert ^{p-2}\nabla u_{p}\cdot\nabla\rho\mathrm{d}x\\
&  \leq\left\Vert \nabla u_{p}\right\Vert _{p}^{p-1}\left\Vert \nabla
\rho\right\Vert _{p}=\left(  \Lambda_{p}(\Omega)u_{p}(x_{0})\right)
^{\frac{p-1}{p}}\left\vert \Omega\right\vert ^{\frac{1}{p}},
\end{align*}
we have%
\[
\rho(x_{0})\leq\Lambda_{p}(\Omega)^{-\frac{1}{p}}\left(  u_{p}(x_{0})\right)
^{\frac{p-1}{p}}\left\vert \Omega\right\vert ^{\frac{1}{p}}.
\]
Hence, after making $p\rightarrow\infty,$ we obtain%
\[
\rho(x_{0})\leq\left\Vert \rho\right\Vert _{\infty}U_{\infty}(x_{0}).
\]
The second inequality in (\ref{aux20}) then implies that $\rho(x_{0}%
)=\left\Vert \rho\right\Vert _{\infty}U_{\infty}(x_{0}).$ Thus, if $\rho
(x_{0})=\left\Vert \rho\right\Vert _{\infty}$ we conclude that $U_{\infty
}(x_{0})=1.$Therefore, (\ref{minf}) holds for $U_{\infty},$ since%
\[
\left\Vert \nabla U_{\infty}\right\Vert _{\infty}\leq\dfrac{1}{\left\Vert
\rho\right\Vert _{\infty}}\leq\frac{\left\Vert \nabla U_{\infty}\right\Vert
_{\infty}}{\left\Vert U_{\infty}\right\Vert _{\infty}}=\left\Vert \nabla
U_{\infty}\right\Vert _{\infty}.\text{ }%
\]

Repeating the arguments in the proof of Theorem \ref{diric} we can check that
$U_{\infty}$ is a viscosity solution of (\ref{infdiric1}), so that $U_{\infty
}=u_{\infty}$ and (\ref{minf}) holds true.
\end{proof}

The following corollary is an immediate consequence of Theorem
\ref{distmaxext}.

\begin{corollary}
Suppose the distance function of $\Omega$ has infinitely many maximum points.
Then, there exist infinitely many positive functions $u\in C(\overline{\Omega
})\cap W_{0}^{1,\infty}(\Omega)$ satisfying
\begin{equation}
0<u(x)\leq1=\left\Vert u\right\Vert _{\infty}\text{ \ and \ }\left\Vert \nabla
u\right\Vert _{\infty}=\min\left\{  \left\Vert \nabla\phi\right\Vert _{\infty
}:\phi\in W_{0}^{1,\infty}(\Omega)\text{ \ and \ }\left\Vert \phi\right\Vert
_{\infty}=1\right\}  . \label{extremin}%
\end{equation}
Moreover, each one of these functions assumes its maximum value $1$ only at
one point, which is also a maximum point of the distance function $\rho.$

In particular, there exist infinitely many nonradial functions satisfying
(\ref{extremin}) for the annulus $\Omega_{a,b}:=\left\{  x\in\mathbb{R}%
^{N}:0<a<\left\vert x\right\vert <b\right\}  .$
\end{corollary}

\section{Acknowledgments}

The first author thanks the support of Funda\c{c}\~{a}o de Amparo \`{a}
Pesquisa do Estado de Minas Gerais (Fapemig)/Brazil (CEX-PPM-00165) and
Conselho Nacional de Desenvolvimento Cient\'{\i}fico e Tecnol\'{o}gico
(CNPq)/Brazil (483970/2013-1 and 306590/2014-0).

\end{document}